\documentclass[preprint,12pt]{elsarticle}

\usepackage{amsmath,amssymb,amsfonts}
\usepackage{amsthm}
\usepackage{graphicx}
\usepackage{booktabs}
\usepackage{array}
\usepackage{xcolor}
\usepackage[linesnumbered,ruled,vlined]{algorithm2e}
\usepackage[colorlinks=true, allcolors=blue]{hyperref}
\usepackage{tikz}
\usetikzlibrary{arrows.meta, shapes.geometric}

\theoremstyle{plain}
\newtheorem{theorem}{Theorem}
\newtheorem{lemma}{Lemma}
\newtheorem{proposition}{Proposition}

\DeclareMathOperator*{\argmin}{argmin}
\DeclareMathOperator*{\argmax}{argmax}

\journal{Computers \& Chemical Engineering}

\begin{document}

\begin{frontmatter}

\title{G-ORACLE: A Fast Geometric Method To Explore All Near-Optimal Designs for Energy Systems}

\author[eins]{Sina Hajikazemi\corref{cor1}}
\ead{sina.hajikazemi@eins.tu-darmstadt.de}
\author[eins]{Florian Steinke}

\cortext[cor1]{Corresponding author}
\address[eins]{Energy Information Networks and Systems, Technical University of Darmstadt, Darmstadt, Germany}

\begin{abstract}
Modeling to Generate Alternatives (MGA) methods explore the near-optimal solutions of energy system optimization models (ESOMs), i.e., alternative planning pathways whose cost lies within a small predefined slack of the optimum but which can differ substantially in technology composition, regional distribution, and environmental impacts.
Among MGA methods, ORACLE uniquely guarantees convergence to the near-optimal set within a targeted tolerance by iteratively refining inner and outer polytope approximations.
At each iteration, given a trial point, it solves a problem as hard as the ESOM to find the nearest near-optimal solution to it.
However, the trial point is obtained by solving a bilevel problem whose solution time grows rapidly with the number of iterations and eventually dominates the computational cost of the method.

To eliminate this bottleneck, we present G-ORACLE, which exploits the fact that an optimal solution of the bilevel problem is always attained at a vertex of the outer approximation.
G-ORACLE therefore replaces the bilevel problem with incremental vertex enumeration by the double description method and one convex projection per vertex, while retaining the convergence guarantee of ORACLE.
A supporting-hyperplane check further avoids unnecessary recomputation of these projections as the inner approximation grows.

We evaluate G-ORACLE on an electricity-only PyPSA-Eur model, exploring the near-optimal space of the installed capacities of six generation and two storage technologies.
While ORACLE's bilevel problem exceeds 20 minutes at iteration 64, G-ORACLE performs the trial-point search in at most 14 seconds throughout 240 iterations, leaving the ESOM solves as the only significant computational cost.
\end{abstract}


\end{frontmatter}

\section{Introduction}\label{sec:intro}

Energy system optimization models (ESOMs) determine the optimal capacity, technology mix, and operation of energy infrastructure by minimizing total system costs subject to technical, environmental, and regulatory constraints.
They serve as key decision-support tools for long-term planning, ranging from macro-level decarbonization pathways to the granular sizing and siting of generation, storage, and transmission assets.

However, cost-optimality is strictly bounded by the assumptions and constraints built into the optimization formulation. Qualitative or hard-to-quantify considerations, such as political feasibility, social acceptance, and localized environmental impacts, are rarely fully represented in the objective function or constraints. Consequently, a mathematically optimal solution may prove undesirable or infeasible in real-world decision-making. This divergence stems from two distinct sources of uncertainty: parametric uncertainty in inputs such as demand, renewable availability, and fuel prices, and structural uncertainty arising from model abstractions and idealized system assumptions. Modeling to Generate Alternatives (MGA) addresses these limitations by shifting focus to the \emph{near-optimal set}, i.e., all feasible solutions whose cost lies within a predefined slack of the optimal cost. This enables decision-makers to systematically evaluate alternative planning pathways against unmodeled criteria \citep{DECAROLIS2011145}.

This capability has made MGA a widely used tool in energy system studies. \citet{trutnevyte2016} demonstrated that the historical development of the Swiss electricity system diverged substantially from cost-optimal trajectories, arguing against reliance on a single optimum. Applying MGA to the global TIAM-UCL model, \citet{price_modelling_2017} identified technologically distinct decarbonization pathways with nearly the same total cost under climate policy constraints. Using the open PyPSA-Eur model, \citet{neumann2021,neumann2023} showed that the near-optimal set contains a technologically diverse range of renewable power system designs, leaving policymakers considerable freedom in planning, and that a subset of these designs remains robust to uncertain technology costs. Building on this work, \citet{grochowicz2023} intersected near-optimal sets across historical weather years to identify system configurations that are robust to weather variability. Other studies used MGA to examine the spatial distribution of infrastructure. \citet{sasse2019} exposed trade-offs between cost-efficiency and regional equity in the near-optimal set, and \citet{lombardi2020} introduced a spatially explicit MGA variant to identify alternative locations for renewable deployment. Finally, \citet{schwaeppe2024} derived near-optimal shadow prices to reveal cost-effective flexibility options. Collectively, these studies demonstrate the value of exploring near-optimal sets and highlight the need for algorithms capable of doing so reliably.

ESOMs contain far more variables than are of direct interest to decision-makers, so MGA studies typically focus on a few aggregated properties of interest (PoIs), such as the total installed capacity of a technology. Since the near-optimal set cannot be derived analytically, most MGA methods approximate it by repeatedly solving the ESOM with an additional constraint that bounds its cost, each time finding the near-optimal solution furthest in a given search direction in the space of PoIs (Section~\ref{sec:model-mga}). Each of these problems is as hard to solve as the original ESOM, so the number of solves required for an adequate exploration largely determines the computational effort of MGA, a trade-off examined by \citet{lombardi2023}. Heuristic MGA methods select the search directions by ad-hoc rules: Variable Min/Max (VMM) \citep{trutnevyte2012a,nacken2019,neumann2021,neumann2023} maximizes and minimizes each PoI; Random Vector (RV) variations \citep{chang1982a,trutnevyte2016,berntsen2017,sasse2019,sasse2023a} draw the objective weights at random and converge to the near-optimal set with probability one as the number of iterations tends to infinity; Hop-Skip-Jump (HSJ) \citep{brill1979,chang1983fuzzy,DECAROLIS2011145} minimizes the sum of the variables that were nonzero in previously found solutions, pushing the next solution away from them; and Manhattan MGA (MMGA) \citep{price_modelling_2017} maximizes the Manhattan distance to previously found solutions. However, these methods provide no stopping criterion, leaving the number of iterations to be set by the user, and no measure of how much of the near-optimal set has been explored. Modeling All Alternatives (MAA) \citep{pedersen2021modeling,grochowicz2023} instead uses the Quickhull algorithm \citep{barber1996quickhull} to compute the facets and the volume of the convex hull of the points found so far. The facet normals serve as new search directions; since the near-optimal set has finitely many vertices, this procedure recovers it in a finite number of iterations \citep{pedersen2021modeling}. In practice, the method can be stopped once the volume increase per iteration falls below a given threshold, which suggests that little of the set remains undiscovered. Since MAA maintains only an inner approximation, this criterion does not bound the unexplored part of the set; moreover, the computational cost of the convex hull grows exponentially with the number of PoIs, which in practice limits MAA to at most about 10 PoIs \citep{pedersen2021modeling}. A detailed review of these methods is given by \citet{Lau_2024}.

ORACLE \citep{turan2026oracle} is, to our knowledge, the only MGA method that guarantees convergence to within a desired tolerance of the near-optimal set, by maintaining an inner and an outer polytope approximation of it and iteratively refining both. Instead of the near-optimal solution furthest in a direction, ORACLE seeks the near-optimal solution nearest to a trial point, which it selects at each iteration as the point of the outer approximation furthest from the inner one. Selecting this trial point requires solving a bilevel problem at every iteration, which ORACLE reformulates as a mixed-integer linear program (MILP).
Linear bilevel programs are strongly NP-hard in general \citep{hansen1992bilevel}, and the number of binary variables in this MILP grows with the number of points in the inner approximation, that is, with every iteration.
As shown in Section~\ref{sec:numerical}, the solution time of this bilevel problem grows rapidly with the number of iterations, whereas that of the ESOM does not, so the bilevel problem eventually dominates the computational cost of the method.

In this paper, we introduce G-ORACLE, which avoids this bilevel problem entirely. We prove that an optimal solution of the bilevel problem is always attained at a vertex of the outer approximation, which lets us replace the bilevel problem with incremental vertex enumeration and independent, single-level convex projections. A supporting-hyperplane check further reuses previously computed distances instead of recomputing them at every iteration. Table~\ref{tab:mga-methods} compares G-ORACLE with the existing MGA methods.

\begin{table}[htbp]
\centering
\caption{Comparison of MGA method categories.}
\label{tab:mga-methods}
\scriptsize
\setlength{\tabcolsep}{4pt}
\begin{tabular}{>{\raggedright\arraybackslash}p{2.2cm}>{\raggedright\arraybackslash}p{2.0cm}>{\raggedright\arraybackslash}p{2.6cm}>{\raggedright\arraybackslash}p{2.7cm}>{\raggedright\arraybackslash}p{2.6cm}}
\toprule[0.1em]
Method & Exploration strategy & Selection mechanism & Stopping criterion & Main limitation \\
\midrule[0.07em]
Heuristics (VMM, RV, HSJ, MMGA) & Search direction & Ad-hoc rules & Iteration budget set by the user & No measure of exploration progress \\
\midrule[0.03em]
MAA & Search direction & Facet normals of the inner convex hull & Volume increase below a threshold (no~error~bound) & Convex-hull cost grows exponentially with the number of PoIs ($\lesssim 10$) \\
\midrule[0.03em]
ORACLE & Trial-point projection & Bilevel problem over the outer approximation & Inner--outer approximation distance $\le \delta$ (error~bound) & MILP reformulation grows with the number of inner points \\
\midrule[0.03em]
\mbox{G-ORACLE} (this work) & Trial-point projection & Vertex enumeration of the outer approximation & Inner--outer approximation distance $\le \delta$ (error~bound) & Number of vertices of the outer approximation \\
\bottomrule[0.1em]
\end{tabular}
\end{table}

Section~\ref{sec:background} formalizes the ESOM and the MGA problem and reviews the ORACLE method. Section~\ref{sec:geometric-replacement} presents G-ORACLE, which replaces ORACLE's bilevel problem with vertex enumeration and per-vertex convex projections, together with a supporting-hyperplane check. Section~\ref{sec:numerical} evaluates G-ORACLE on a PyPSA-Eur model and reports the resulting speedup over ORACLE. Section~\ref{sec:conclusion} concludes and discusses directions for future work.

\section{Background}\label{sec:background}
In this section we formalize the problem and review the ORACLE method, on which G-ORACLE builds.

\subsection{Problem Formulation}\label{sec:model-mga}

Due to their large size, ESOMs are typically formulated as linear programs (LPs) to ensure computational tractability \citep{kotzur2021modelers,cho2022expansion}. We consider a linear ESOM in the general form
\begin{equation}\label{eq:mp}
\begin{aligned}[t]
\nu^* = \min_{x \in \mathbb{R}^{n_x}} \quad & c^\top x \\
\text{s.t.} \quad & Ax \le b,
\end{aligned}
\end{equation}
where $x \in \mathbb{R}^{n_x}$ denotes the model's planning and operational decision variables, $c$ is the cost vector, and $A, b$ define the technical and economic constraints of the system.
Let $x^*$ denote an optimal solution to \eqref{eq:mp} with minimum total cost $\nu^* = c^\top x^*$, which we assume to be positive.

For a cost slack $\epsilon \ge 0$, the near-optimal solutions of \eqref{eq:mp} form the polyhedron
\begin{equation}\label{eq:near-opt}
\mathcal{X}^\epsilon = \{ x \in \mathbb{R}^{n_x} \mid Ax \le b, \; c^\top x \le (1+\epsilon)\nu^* \}.
\end{equation}
Because $x$ typically contains far more variables than are of direct interest to a decision-maker, MGA methods focus on a small number of PoIs.
These are defined via a linear mapping $z = Dx$, where $D \in \mathbb{R}^{n_z \times n_x}$ with $n_z \ll n_x$ (e.g., aggregating installed capacities of key technologies).
We define the \emph{near-optimal set} as the image of $\mathcal{X}^\epsilon$ under this mapping,
\begin{equation}\label{eq:near-opt-poi}
\mathcal{Z} = \{ z \in \mathbb{R}^{n_z} \mid z = Dx, \; x \in \mathcal{X}^\epsilon \}.
\end{equation}
Since $\mathcal{X}^\epsilon$ is a polyhedron and the map $x \mapsto Dx$ is linear, $\mathcal{Z}$ is also a polyhedron. We assume throughout that $\mathcal{Z}$ is bounded, so that it is a polytope.

Computing $\mathcal{Z}$ explicitly requires projecting the polyhedron $\mathcal{X}^\epsilon \subset \mathbb{R}^{n_x}$ onto $\mathbb{R}^{n_z}$, e.g., by Fourier--Motzkin elimination, which is intractable for ESOMs of realistic size. Direction-based MGA methods therefore approximate $\mathcal{Z}$ by iteratively exploring its vertices using Problem~\eqref{eq:mga-subproblem}, which maximizes a linear objective $w^\top z$ over $\mathcal{Z}$ and thus returns the near-optimal solution furthest in a given search direction $w \in \mathbb{R}^{n_z}$.

\begin{equation}\label{eq:mga-subproblem}
\begin{aligned}[t]
\max_{x \in \mathbb{R}^{n_x}} \quad & w^\top Dx \\
\text{s.t.} \quad & Ax \le b, \\
& c^\top x \le (1+\epsilon)\nu^*.
\end{aligned}
\end{equation}
Each solve of Problem~\eqref{eq:mga-subproblem} yields a near-optimal point $z = Dx$.

\subsection{The ORACLE Method}\label{sec:oracle-method}

ORACLE is an iterative method that maintains an inner approximation $\mathcal{I}$ and an outer approximation $\mathcal{O}$ of the near-optimal set, with $\mathcal{I} \subseteq \mathcal{Z} \subseteq \mathcal{O}$. The inner approximation is the convex hull of a finite set of known points $\mathcal{V}_I \subset \mathcal{Z}$, while the outer approximation is a polytope defined as the intersection of finitely many halfspaces each satisfied by every point of $\mathcal{Z}$.
At each iteration, the inner approximation expands by adding a new point of $\mathcal{Z}$ to $\mathcal{V}_I$, while the outer approximation shrinks by adding a halfspace that cuts off part of $\mathcal{O}$.
The algorithm terminates once the directed Hausdorff distance from the outer to the inner approximation falls below a desired convergence tolerance $\delta > 0$, which guarantees that both approximations are within that tolerance of the true near-optimal set $\mathcal{Z}$ in the same sense.

We summarize the four steps of the algorithm below. Throughout, we fix a norm $\|\cdot\|$ to measure distances in $\mathbb{R}^{n_z}$. Given a point $z \in \mathbb{R}^{n_z}$, its distance to the inner approximation $\mathcal{I}$ is given by
\begin{equation}\label{eq:f-def}
\begin{aligned}[t]
f(z) = \min_{y,\, s \in \mathbb{R}^{n_z}} \quad & \|s\| \\
\text{s.t.} \quad & y \in \mathcal{I}, \\
& z - y = s, \quad (\pi)
\end{aligned}
\end{equation}
where $\pi$ is the dual vector of the residual constraint.

\textbf{Step 1: Initialization.} Before the main loop, Problem~\eqref{eq:mga-subproblem} is solved for $w = \pm e_i$, $i = 1, \dots, n_z$, the standard basis vectors and their negatives, which maximizes and minimizes each PoI over $\mathcal{Z}$.
The resulting $2n_z$ points, together with the cost-optimal point $z^* = Dx^*$, are added to $\mathcal{V}_I$, initializing $\mathcal{I}$, and the initial outer approximation is the axis-aligned box $\mathcal{O} = \{z \in \mathbb{R}^{n_z} \mid z_{lb} \le z \le z_{ub}\}$, where $z_{lb}$ and $z_{ub}$ collect the resulting minima and maxima.

\textbf{Step 2: Trial point.} The point of $\mathcal{O}$ furthest from $\mathcal{I}$ is found by solving
\begin{equation}\label{eq:trial-point-problem}
z_{O} \in \argmax_{z \in \mathcal{O}} f(z), \qquad d = f(z_{O}).
\end{equation}
By construction, $d$ is the directed Hausdorff distance from $\mathcal{O}$ to $\mathcal{I}$. Because $\mathcal{I} \subseteq \mathcal{Z} \subseteq \mathcal{O}$, $d$ is also an upper bound on the directed Hausdorff distances from $\mathcal{O}$ to $\mathcal{Z}$ and from $\mathcal{Z}$ to $\mathcal{I}$. The algorithm terminates if $d \le \delta$.

\textbf{Step 3: Nearest feasible point.} If $d > \delta$, the nearest feasible point $z_I = Dx_I \in \mathcal{Z}$ to $z_O$ is found by solving
\begin{equation}\label{eq:nearest-feasible}
\begin{aligned}[t]
x_I \in \argmin_{x \in \mathbb{R}^{n_x},\, r \in \mathbb{R}^{n_z}} \quad & \|r\| \\
\text{s.t.} \quad & Ax \le b, \\
& c^\top x \le (1+\epsilon)\nu^*, \\
& z_O - Dx = r, \quad (\sigma)
\end{aligned}
\end{equation}
where $\sigma$ is the dual vector of the residual constraint.
In contrast to Problem~\eqref{eq:mga-subproblem}, which returns the near-optimal solution furthest in a search direction $w$, Problem~\eqref{eq:nearest-feasible} seeks the nearest one to $z_O$.

\textbf{Step 4: Approximation update.} $\mathcal{I}$ is expanded by adding $z_I$ to $\mathcal{V}_I$. If $z_I \ne z_O$, the outer approximation is updated as $\mathcal{O} \gets \mathcal{O} \cap \mathcal{H}_\mathrm{P}$, where
\begin{equation}\label{eq:projection-cut}
\mathcal{H}_\mathrm{P} = \{z \in \mathbb{R}^{n_z} \mid \sigma^\top z \le \sigma^\top z_I\}
\end{equation}
is the \emph{projection cut}, which contains $\mathcal{Z}$ and cuts off $z_O$ \citep[Appendix~A.4.1]{turan2026oracle}.

A second cut, the \emph{cost cut} \citep[Appendix~A.4.2]{turan2026oracle}, can optionally be added to $\mathcal{O}$.
It is derived from the value function
\begin{equation}\label{eq:value-fn}
\begin{aligned}[t]
\nu(z) = \min_{x \in \mathbb{R}^{n_x}} \quad & c^\top x \\
\text{s.t.} \quad & Ax \le b, \\
& Dx = z, \quad (\mu)
\end{aligned}
\end{equation}
that is, the minimum total system cost for given values $z$ of the PoIs, with $\nu(z) = +\infty$ if Problem~\eqref{eq:value-fn} is infeasible.
Since $\nu$ is convex \citep[Section~5.6]{boyd2004convex}, the dual vector $\mu$ of the constraint $Dx = z$ in Problem~\eqref{eq:value-fn} with $z = z_I$ is a subgradient of $\nu$ at $z_I$, so $\nu(z) \ge \nu(z_I) + \mu^\top(z - z_I)$ for all $z$.
Moreover, every $z \in \mathcal{Z}$ satisfies $\nu(z) \le (1+\epsilon)\nu^*$, since $z = Dx$ for some $x \in \mathcal{X}^\epsilon$, which is feasible in Problem~\eqref{eq:value-fn}.
Hence the cost cut
\begin{equation}\label{eq:cost-cut}
\mathcal{H}_\mathrm{C} = \{z \in \mathbb{R}^{n_z} \mid \nu(z_I) + \mu^\top(z - z_I) \le (1+\epsilon)\nu^*\}
\end{equation}
contains $\mathcal{Z}$.
The cost cut is not required for convergence, but it can accelerate convergence by removing parts of $\mathcal{O}$ that violate the near-optimality constraint.
We therefore omit it from Algorithms~\ref{alg:oracle} and~\ref{alg:geometric-trial-point}, but use it in the numerical experiments (Section~\ref{sec:numerical}).

This completes ORACLE's four steps, summarized in Algorithm~\ref{alg:oracle} (Steps~2--4) and illustrated in Figure~\ref{fig:oracle-illustration}.
However, Problem~\eqref{eq:trial-point-problem} in Step~2 is a bilevel problem. The upper level maximizes $f(z)$ over $z \in \mathcal{O}$, while evaluating $f(z)$ requires solving Problem~\eqref{eq:f-def} as the lower level, which minimizes the distance from $z$ to $\mathcal{I}$.
It is reformulated as a single-level MILP by replacing the lower-level problem with its KKT conditions. The complementarity constraints are linearized using the big-M method \citep{bard1998bilevel,pineda2018bilevel} that introduces $O(|\mathcal{V}_I|)$ binary variables.
To keep this reformulation a MILP, $\|\cdot\|$ is restricted to the $\ell_1$ or $\ell_\infty$ norm.\footnote{For the $\ell_2$ norm, the lower-level problem can be written as a convex QP, and its KKT conditions can be linearized in the same way; however, the upper level then maximizes a convex quadratic function, so the reformulation becomes a nonconvex mixed-integer quadratic program rather than a MILP.} As shown in Section~\ref{sec:numerical}, the time needed to solve this bilevel problem grows rapidly with the number of inner points.

\begin{algorithm}
\caption{ORACLE \citep{turan2026oracle}}\label{alg:oracle}
\KwData{Initial $\mathcal{I} = \operatorname{conv}(\mathcal{V}_I)$, $\mathcal{O}$ (Step 1); tolerance $\delta > 0$}
\KwResult{$\mathcal{I}, \mathcal{O}$ converged to within $\delta$ of $\mathcal{Z}$}
\Repeat{$d \le \delta$}{
    $z_O \gets \argmax_{z \in \mathcal{O}} f(z)$;\quad $d \gets f(z_O)$ \tcp*{Step 2}
    \If{$d > \delta$}{
        solve Problem~\eqref{eq:nearest-feasible} and obtain $z_I$, $\sigma$ \tcp*{Step 3}
        $\mathcal{V}_I \gets \mathcal{V}_I \cup \{z_I\}$ \tcp*{Step 4}
        \If{$z_I \ne z_O$}{
            $\mathcal{O} \gets \mathcal{O} \cap \mathcal{H}_\mathrm{P}$\;
        }
    }
}
\end{algorithm}

\begin{figure}
\centering
\resizebox{\textwidth}{!}{%
\begin{tikzpicture}[
  truez/.style={dashed, thick, black!55},
  outerbox/.style={thick, blue!65!black},
  gapfill/.style={fill=blue!8},
  innerfill/.style={fill=green!12},
  ivertex/.style={circle, fill=green!55!black, inner sep=2.2pt},
  ptz/.style={circle, fill=red, inner sep=2.2pt},
  pti/.style={circle, fill=orange!90!black, inner sep=2.2pt},
  distline/.style={black!70, thick},
  paneltitle/.style={font=\small\bfseries},
  >=Stealth
]

\begin{scope}[scale=0.95]
\fill[gapfill] (-1.8,-2.2) -- (2.4,-2.2) -- (2.4,2.0) -- (-1.8,2.0) -- cycle;
\fill[innerfill] (2.4,0.3) -- (0.2,2.0) -- (-1.8,0.6) -- (0.9,-2.2) -- cycle;
\draw[truez] (2.4,0.3) -- (1.6,1.7) -- (0.2,2.0) -- (-1.8,0.6) -- (-1.0,-1.6) -- (0.9,-2.2) -- cycle;
\draw[outerbox] (-2.1,-2.2) -- (2.7,-2.2);
\draw[outerbox] (2.4,-2.5) -- (2.4,2.3);
\draw[outerbox] (-2.1,2.0) -- (2.7,2.0);
\draw[outerbox] (-1.8,-2.5) -- (-1.8,2.3);
\foreach \p in {(2.4,0.3),(0.2,2.0),(-1.8,0.6),(0.9,-2.2)} \node[ivertex] at \p {};
\node[green!40!black] at (0.4,0.3) {$\mathcal{I}$};
\node[blue!65!black] at (-1.3,1.5) {$\mathcal{O}$};
\node[black] (zlabel) at (1.9,-1.6) {$\mathcal{Z}$};
\draw[->, thin, black] (zlabel) -- (1.65,-0.95);
\node[paneltitle] at (0.3,-2.8) {Step 1};
\end{scope}

\begin{scope}[shift={(5.5,0)}, scale=0.95]
\fill[gapfill] (-1.8,-2.2) -- (2.4,-2.2) -- (2.4,2.0) -- (-1.8,2.0) -- cycle;
\fill[innerfill] (2.4,0.3) -- (0.2,2.0) -- (-1.8,0.6) -- (0.9,-2.2) -- cycle;
\draw[truez] (2.4,0.3) -- (1.6,1.7) -- (0.2,2.0) -- (-1.8,0.6) -- (-1.0,-1.6) -- (0.9,-2.2) -- cycle;
\draw[outerbox] (-2.1,-2.2) -- (2.7,-2.2);
\draw[outerbox] (2.4,-2.5) -- (2.4,2.3);
\draw[outerbox] (-2.1,2.0) -- (2.7,2.0);
\draw[outerbox] (-1.8,-2.5) -- (-1.8,2.3);
\foreach \p in {(2.4,0.3),(0.2,2.0),(-1.8,0.6),(0.9,-2.2)} \node[ivertex] at \p {};
\draw[distline] (-1.8,-2.2) -- (-0.401,-0.850) node[pos=0.4, sloped, above, inner sep=2pt] {$d$};
\draw[thin, black!50] ([shift={(-1.8,-2.2)}]24:1.9442) arc (24:64:1.9442);
\node[ptz] at (-0.401,-0.850) {};
\node[ptz, label={[xshift=3pt]below left:$z_O$}] at (-1.8,-2.2) {};
\node[paneltitle] at (0.3,-2.8) {Step 2};
\end{scope}

\begin{scope}[shift={(11.0,0)}, scale=0.95]
\fill[gapfill] (-1.8,-2.2) -- (2.4,-2.2) -- (2.4,2.0) -- (-1.8,2.0) -- cycle;
\fill[innerfill] (2.4,0.3) -- (0.2,2.0) -- (-1.8,0.6) -- (0.9,-2.2) -- cycle;
\draw[truez] (2.4,0.3) -- (1.6,1.7) -- (0.2,2.0) -- (-1.8,0.6) -- (-1.0,-1.6) -- (0.9,-2.2) -- cycle;
\draw[outerbox] (-2.1,-2.2) -- (2.7,-2.2);
\draw[outerbox] (2.4,-2.5) -- (2.4,2.3);
\draw[outerbox] (-2.1,2.0) -- (2.7,2.0);
\draw[outerbox] (-1.8,-2.5) -- (-1.8,2.3);
\foreach \p in {(2.4,0.3),(0.2,2.0),(-1.8,0.6),(0.9,-2.2)} \node[ivertex] at \p {};
\draw[distline] (-1.8,-2.2) -- (-1.0,-1.6);
\draw[thin, black!50] ([shift={(-1.8,-2.2)}]17:1.0) arc (17:57:1.0);
\node[ptz, label={[xshift=3pt]below left:$z_O$}] at (-1.8,-2.2) {};
\node[pti, label={below:$z_I$}] at (-1.0,-1.6) {};
\node[paneltitle] at (0.3,-2.8) {Step 3};
\end{scope}

\begin{scope}[shift={(16.5,0)}, scale=0.95]
\fill[gapfill] (2.4,-2.2) -- (-0.55,-2.2) -- (-1.8,-0.5333) -- (-1.8,2.0) -- (2.4,2.0) -- cycle;
\fill[innerfill] (2.4,0.3) -- (0.2,2.0) -- (-1.8,0.6) -- (-1.0,-1.6) -- (0.9,-2.2) -- cycle;
\draw[truez] (2.4,0.3) -- (1.6,1.7) -- (0.2,2.0) -- (-1.8,0.6) -- (-1.0,-1.6) -- (0.9,-2.2) -- cycle;
\draw[outerbox] (2.7,-2.2) -- (-0.85,-2.2);
\draw[outerbox] (-0.37,-2.44) -- (-1.98,-0.2933);
\draw[outerbox] (-1.8,-0.8333) -- (-1.8,2.3);
\draw[outerbox] (-2.1,2.0) -- (2.7,2.0);
\draw[outerbox] (2.4,2.3) -- (2.4,-2.5);
\foreach \p in {(2.4,0.3),(0.2,2.0),(-1.8,0.6),(0.9,-2.2),(-1.0,-1.6)} \node[ivertex] at \p {};
\node[paneltitle] at (0.3,-2.8) {Step 4};
\end{scope}

\draw[->, thick] (2.55,-0.1) -- (3.55,-0.1);
\draw[->, thick] (13.55,-0.1) -- (14.55,-0.1);

\node[diamond, draw, thick, fill=white, minimum width=1.6cm, minimum height=0.9cm, font=\footnotesize, align=center] (dec) at (5.8,-4.6) {$d \le \delta$?};
\draw[->, thick] (5.8,-3.0) -- (dec.north);
\draw[->, thick] (dec.west) -- node[above, font=\footnotesize] {yes} ++(-1.3,0) node[left, font=\footnotesize] {terminate};
\draw[->, thick] (dec.east) -- node[above, font=\footnotesize] {no} (11.3,-4.6) -- (11.3,-3.0);

\draw[->, thick]
  (16.8,2.3) -- ++(0,0.7)
  -- (5.8,3.0)
  -- (5.8,2.3);

\end{tikzpicture}
}
\caption{
    Illustration of ORACLE algorithm. The near-optimal set $\mathcal{Z}$ (dashed), $\mathcal{O}$ in blue, $\mathcal{I}$ in green.
    Step 1: both approximations are initialized by maximizing and minimizing each coordinate of $\mathbb{R}^{n_z}$ over $\mathcal{Z}$.
    Step 2: the trial point $z_O$, the member of $\mathcal{O}$ furthest from $\mathcal{I}$, is found by solving a bilevel problem.
    Step 3: the nearest point $z_I$ of $\mathcal{Z}$ to $z_O$ is found.
    Step 4: a halfspace cutting off $z_O$ shrinks $\mathcal{O}$, and $\mathcal{I}$ expands by adding $z_I$ to it.
    }
\label{fig:oracle-illustration}
\end{figure}
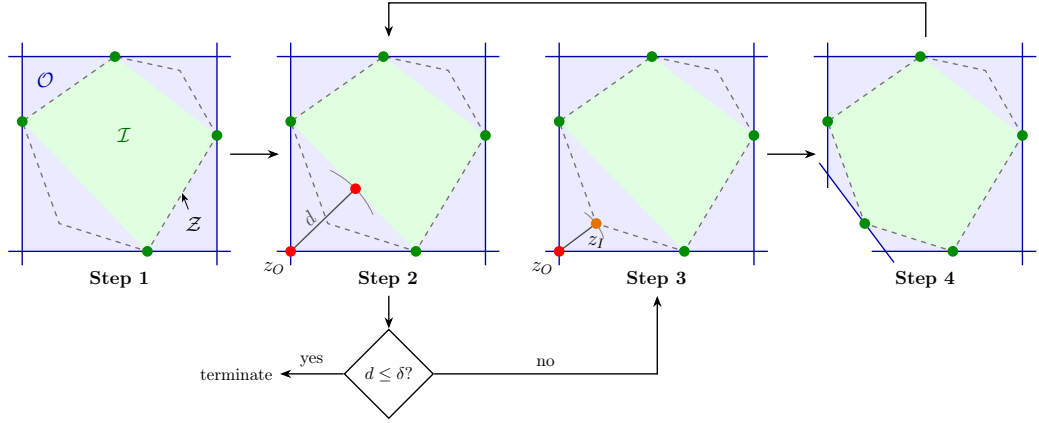

\section{G-ORACLE}\label{sec:geometric-replacement}

In this section we show that, since $f$ is convex (Lemma~\ref{lemma:convex-dist}), Problem~\eqref{eq:trial-point-problem} always has an optimal solution at a vertex of the outer approximation $\mathcal{O}$.
Given this result, the bilevel problem can be replaced with a procedure that (a) enumerates the vertex set $\mathcal{V}_O = \operatorname{vert}(\mathcal{O})$, (b) evaluates $f(v)$ at each vertex $v \in \mathcal{V}_O$, and (c) returns a vertex with the largest value of $f$.
Step (b) requires solving an independent single-level convex program~\eqref{eq:f-def} for each vertex, which can be formulated as an LP for the $\ell_1$ and $\ell_\infty$ norms, or as a QP for the $\ell_2$ norm.
Section~\ref{sec:vertex-enum} establishes the vertex property and describes the incremental vertex enumeration.
Section~\ref{sec:distance-update} introduces a certificate that identifies vertices whose previously computed distance remains valid as $\mathcal{I}$ grows, so that $f$ need not be re-evaluated at them.
Section~\ref{sec:algorithm} summarizes the resulting algorithm.

\subsection{Vertex Enumeration}\label{sec:vertex-enum}

We first show that the optimum of Problem~\eqref{eq:trial-point-problem} is attained at a vertex of $\mathcal{O}$, and then describe how $\mathcal{V}_O$ is maintained as $\mathcal{O}$ is refined.
The vertex property follows from the convexity of $f$.

\begin{lemma}\label{lemma:convex-dist}
For any non-empty, closed, convex set $\mathcal{I} \subseteq \mathbb{R}^{n_z}$ and any norm $\|\cdot\|$ on $\mathbb{R}^{n_z}$, the function $f$ is convex \citep[Section~3.2.5]{boyd2004convex}.
\end{lemma}
\begin{proof}
Let $z_1, z_2 \in \mathbb{R}^{n_z}$ and $\lambda \in [0,1]$. Since $\mathcal{I}$ is non-empty and closed, there exist $y_1, y_2 \in \mathcal{I}$ achieving $f(z_1) = \|z_1-y_1\|$ and $f(z_2) = \|z_2-y_2\|$.

By convexity of $\mathcal{I}$, $y_\lambda = \lambda y_1 + (1-\lambda)y_2 \in \mathcal{I}$. By the definition of $f$, the triangle inequality, and positive homogeneity of the norm,
\begin{align*}
f(\lambda z_1 + (1-\lambda)z_2) &\le \|\lambda z_1 + (1-\lambda)z_2 - y_\lambda\| \\
&= \|\lambda(z_1-y_1) + (1-\lambda)(z_2-y_2)\| \\
&\le \lambda \|z_1-y_1\| + (1-\lambda) \|z_2-y_2\| \\
&= \lambda f(z_1) + (1-\lambda) f(z_2). \qedhere
\end{align*}
\end{proof}

\begin{theorem}\label{thm:vertex}
Let $\mathcal{I} \subseteq \mathbb{R}^{n_z}$ be non-empty, closed, and convex, and let $\mathcal{O} \subseteq \mathbb{R}^{n_z}$ be a non-empty polytope. Then $\max_{z \in \mathcal{O}} f(z)$ is attained at a vertex of $\mathcal{O}$ \citep[cf.][Corollary~32.3.2]{rockafellar1970convex}.
\end{theorem}

\begin{proof}
By the Minkowski--Weyl theorem, $\mathcal{O}$ is the convex hull of its finitely many vertices $\operatorname{vert}(\mathcal{O}) = \{v_1, \dots, v_k\}$, so every $z \in \mathcal{O}$ can be written as $z = \sum_{i=1}^k \lambda_i v_i$ with $\sum_{i=1}^k \lambda_i = 1$ and $\lambda_i \ge 0$.

Let $v^\star \in \argmax_{v \in \operatorname{vert}(\mathcal{O})} f(v)$, which exists since $\operatorname{vert}(\mathcal{O})$ is finite. By Lemma~\ref{lemma:convex-dist} and Jensen's inequality,
\begin{equation*}
f(z) = f\left(\sum_{i=1}^k \lambda_i v_i\right) \le \sum_{i=1}^k \lambda_i f(v_i) \le f(v^\star) \quad \forall z \in \mathcal{O}.
\end{equation*}
Since $v^\star \in \mathcal{O}$, it is a maximizer of $f$ over $\mathcal{O}$.
\end{proof}

Since $\mathcal{I} = \operatorname{conv}(\mathcal{V}_I)$ is a non-empty polytope, Theorem~\ref{thm:vertex} applies to Problem~\eqref{eq:trial-point-problem} for any choice of norm.

Recovering the vertex ($V$-) representation of a polytope from its halfspace ($H$-) representation, here $\mathcal{V}_O$ from the cuts defining $\mathcal{O}$, is a classical problem in polyhedral computation.
Since $\mathcal{O}$ is refined by adding halfspaces, we update $\mathcal{V}_O$ incrementally with the double description method \citep{fukuda1996double}, processing the new halfspaces one at a time rather than recomputing $\mathcal{V}_O$ from scratch.
When a halfspace $\mathcal{H}$ is added, the method partitions the current vertices into those strictly inside $\mathcal{H}$, those on its bounding hyperplane, and those outside it.
The vertices outside $\mathcal{H}$ are discarded, and a new vertex is created where the bounding hyperplane intersects each edge of $\mathcal{O}$ that joins a vertex inside $\mathcal{H}$ to one outside it.

The cost of each double description update, as well as the number of distance evaluations in Step~2, grows with the number of vertices of $\mathcal{O}$.
By the upper bound theorem \citep{mcmullen1970upperbound}, an $n_z$-dimensional polytope defined by $m \ge n_z + 1$ linear inequalities has at most
\begin{equation}\label{eq:upper-bound}
V(m, n_z) = \binom{m - \lceil n_z/2 \rceil}{\lfloor n_z/2 \rfloor} + \binom{m - \lfloor n_z/2 \rfloor - 1}{\lceil n_z/2 \rceil - 1}
\end{equation}
vertices, a bound that is $O\left(m^{\lfloor n_z/2 \rfloor}\right)$ for fixed $n_z$ \citep{seidel1995upper}.
However, in our experiments, the number of vertices remains far below this worst-case bound (Section~\ref{sec:numerical}).

\subsection{Distance Update}\label{sec:distance-update}

Evaluating $f(v)$ at every vertex $v \in \mathcal{V}_O$ in every iteration is largely redundant, because most vertices persist between iterations, and $\mathcal{I}$ grows by a single point at a time.
We therefore derive a sufficient condition, based on the optimal dual vector $\pi$ already obtained when $f(v)$ was computed, under which this distance remains valid after $\mathcal{I}$ is updated, so that Problem~\eqref{eq:f-def} need not be solved again for $v$.

\begin{lemma}\label{lemma:dual-cert}
Let $y^* \in \mathcal{I}$ and $s^* = z - y^*$ be an optimal solution to \eqref{eq:f-def}, and let $\pi \in \mathbb{R}^{n_z}$ be an associated optimal dual vector. Then
\begin{equation}\label{eq:dual-cert-cond}
\pi^\top y \le \pi^\top y^*, \qquad \forall y \in \mathcal{I}.
\end{equation}
\end{lemma}

\begin{proof}
The Lagrangian $\mathcal{L}(y, s, \pi)$ for problem \eqref{eq:f-def} with $y \in \mathcal{I}$ and $s \in \mathbb{R}^{n_z}$ is given by
\begin{equation*}
\mathcal{L}(y, s, \pi) = \|s\| + \pi^\top(z - y - s) = \pi^\top z + (\|s\| - \pi^\top s) - \pi^\top y.
\end{equation*}
Let $\|\pi\|_* = \sup_{\|w\| \le 1} \pi^\top w$ denote the dual norm, for which H\"older's inequality gives $\pi^\top w \le \|\pi\|_* \|w\|$ for every $w \in \mathbb{R}^{n_z}$. Minimizing $\mathcal{L}(y, s, \pi)$ over $s \in \mathbb{R}^{n_z}$ yields $\inf_s (\|s\| - \pi^\top s) = 0$ if $\|\pi\|_* \le 1$, and $-\infty$ otherwise. Restricting to dual-feasible $\pi$ and further minimizing over $y \in \mathcal{I}$ yields the dual objective function
\begin{equation*}
g(\pi) = \inf_{y \in \mathcal{I},\, s \in \mathbb{R}^{n_z}} \mathcal{L}(y, s, \pi) = \pi^\top z - \sup_{y \in \mathcal{I}} \pi^\top y, \qquad \text{for } \|\pi\|_* \le 1.
\end{equation*}

Since the only dualized constraint is linear and \eqref{eq:f-def} is feasible with a finite optimum, strong duality holds, so $f(z) = \|s^*\| = g(\pi)$. Substituting $z = y^* + s^*$ into $g(\pi)$ yields
\begin{equation*}
\|s^*\| = \pi^\top(y^* + s^*) - \sup_{y \in \mathcal{I}} \pi^\top y.
\end{equation*}
Rearranging terms and applying H\"older's inequality to $s^*$ ($\pi^\top s^* \le \|\pi\|_* \|s^*\| \le \|s^*\|$, using dual feasibility $\|\pi\|_* \le 1$), we obtain
\begin{equation*}
\sup_{y \in \mathcal{I}} \pi^\top y - \pi^\top y^* = \pi^\top s^* - \|s^*\| \le 0.
\end{equation*}
Thus, $\sup_{y \in \mathcal{I}} \pi^\top y \le \pi^\top y^*$, which implies $\pi^\top y \le \pi^\top y^*$ for all $y \in \mathcal{I}$.
\end{proof}

Geometrically, \eqref{eq:dual-cert-cond} states that the hyperplane $\{y \in \mathbb{R}^{n_z} \mid \pi^\top y = \pi^\top y^*\}$ supports $\mathcal{I}$ at $y^*$ (Figure~\ref{fig:light-check-illustration}).

\begin{theorem}\label{thm:light-check}
Let $y^* \in \mathcal{I}$ and $\pi \in \mathbb{R}^{n_z}$ be an optimal solution and an associated optimal dual vector of Problem~\eqref{eq:f-def} for a given $z \in \mathbb{R}^{n_z}$. Let $z_I \in \mathbb{R}^{n_z}$ be a newly added point, $\mathcal{V}_I' = \mathcal{V}_I \cup \{z_I\}$, and $\mathcal{I}' = \operatorname{conv}(\mathcal{V}_I')$. If
\begin{equation}\label{eq:light-check}
\pi^\top z_I \le \pi^\top y^*,
\end{equation}
then $y^*$ and $\pi$ remain an optimal solution and an optimal dual vector of Problem~\eqref{eq:f-def} with $\mathcal{I}$ replaced by $\mathcal{I}'$, and in particular $f(z)$ is unchanged.
\end{theorem}

\begin{proof}
By Lemma~\ref{lemma:dual-cert}, the optimality of $y^*$ and $\pi$ over $\mathcal{I}$ implies that condition \eqref{eq:dual-cert-cond} holds for all $y \in \mathcal{I}$. It suffices to show that \eqref{eq:dual-cert-cond} continues to hold for all $y \in \mathcal{I}'$. Any $y \in \mathcal{I}' = \operatorname{conv}(\mathcal{V}_I')$ can be written as $y = \sum_{i=1}^k \lambda_i u_i$ with $u_i \in \mathcal{V}_I' = \mathcal{V}_I \cup \{z_I\}$, $\lambda_i \ge 0$, and $\sum_{i=1}^k \lambda_i = 1$. For every $u_i \in \mathcal{V}_I \subseteq \mathcal{I}$, $\pi^\top u_i \le \pi^\top y^*$ holds by \eqref{eq:dual-cert-cond}; for $u_i = z_I$, it holds by hypothesis \eqref{eq:light-check}. Therefore,
\begin{equation*}
\pi^\top y = \sum_{i=1}^k \lambda_i (\pi^\top u_i) \le \sum_{i=1}^k \lambda_i (\pi^\top y^*) = \pi^\top y^*, \qquad \forall y \in \mathcal{I}'.
\end{equation*}
Dual feasibility $\|\pi\|_* \le 1$ ensures, by H\"older's inequality, that the linear function $y \mapsto \pi^\top(z - y)$ never exceeds the distance $\|z - y\|$; the supporting hyperplane then shows that this lower bound is smallest at $y^*$, where it is tight.
Formally, for any $y \in \mathcal{I}'$,
\begin{equation*}
\|z - y\| \ge \pi^\top (z - y) = \pi^\top z - \pi^\top y \ge \pi^\top z - \pi^\top y^* = \pi^\top (z - y^*).
\end{equation*}
Since $y^*$ and $\pi$ are optimal over $\mathcal{I}$, strong duality gives $\pi^\top (z - y^*) = \|z - y^*\|$.
Hence, $\|z - y\| \ge \|z - y^*\|$ for all $y \in \mathcal{I}'$, proving that $y^*$ remains an optimal solution over $\mathcal{I}'$ and $f(z)$ remains unchanged.
Moreover, $\pi$ remains dual feasible, since $\|\pi\|_* \le 1$ does not depend on $\mathcal{I}$. Since $y^* \in \mathcal{I}'$ and $\pi^\top y \le \pi^\top y^*$ for all $y \in \mathcal{I}'$, we have $\sup_{y \in \mathcal{I}'} \pi^\top y = \pi^\top y^*$, so the dual objective of the problem over $\mathcal{I}'$ at $\pi$ is $\pi^\top z - \pi^\top y^* = \|z - y^*\|$. This equals the primal optimal value, so the duality gap is zero and $\pi$ is dual optimal.
\end{proof}

\begin{figure}
\centering
\resizebox{\textwidth}{!}{%
\begin{tikzpicture}[>=Stealth, scale=0.9]
\begin{scope}
\coordinate (v) at (-0.5,0);
\coordinate (ystar) at (0.5,0);
\fill[green!12] (0.5,0) -- (1.2,1.3) -- (2.8,1.5) -- (3.3,0.2) -- (2.0,-0.6) -- cycle;
\node[green!40!black] at (2.2,0.3) {$\mathcal{I}$};
\draw[dashed, blue!65!black, thick] (0.5,-1.3) -- (0.5,1.7);
\draw[black!70, thick] (v) -- (ystar);
\draw[densely dotted, thick, black!55] (0.5,0) -- (-0.5,1) -- (-1.5,0) -- (-0.5,-1) -- cycle;
\draw[->, thick, blue!65!black] (0.5,-1.15) -- ++(-0.7,0) node[below, midway] {$\pi$};
\draw[dashed, green!55!black, thick] (0.5,0) -- (0.7,1.15) -- (1.2,1.3);
\foreach \p in {(1.2,1.3),(2.8,1.5),(3.3,0.2),(2.0,-0.6)} \filldraw[green!55!black] \p circle (2.2pt);
\filldraw[red] (v) circle (2.2pt);
\node[below] at (v) {$z$};
\filldraw[orange!90!black] (ystar) circle (2.2pt);
\node[right, xshift=2pt, yshift=4pt] at (ystar) {$y^*$};
\filldraw[green!55!black] (0.7,1.15) circle (2.2pt);
\node[above, green!55!black] at (0.7,1.15) {$z_I$};
\node[font=\small\bfseries] at (0.9,-2.6) {(a)};
\end{scope}
\begin{scope}[shift={(5.5,0)}]
\coordinate (v) at (-0.5,0);
\coordinate (ystar) at (0.5,0);
\fill[green!12] (0.5,0) -- (1.2,1.3) -- (2.8,1.5) -- (3.3,0.2) -- (2.0,-0.6) -- cycle;
\node[green!40!black] at (2.2,0.3) {$\mathcal{I}$};
\draw[dashed, blue!65!black, thick] (0.5,-1.3) -- (0.5,1.7);
\draw[black!70, thick] (v) -- (ystar);
\draw[densely dotted, thick, black!55] (0.5,0) -- (-0.5,1) -- (-1.5,0) -- (-0.5,-1) -- cycle;
\draw[->, thick, blue!65!black] (0.5,-1.15) -- ++(-0.7,0) node[below, midway] {$\pi$};
\draw[dashed, green!55!black, thick] (0.5,0) -- (0.25,1.11) -- (1.2,1.3);
\foreach \p in {(1.2,1.3),(2.8,1.5),(3.3,0.2),(2.0,-0.6)} \filldraw[green!55!black] \p circle (2.2pt);
\filldraw[red] (v) circle (2.2pt);
\node[below] at (v) {$z$};
\filldraw[orange!90!black] (ystar) circle (2.2pt);
\node[right, xshift=2pt, yshift=4pt] at (ystar) {$y^*$};
\filldraw[green!55!black] (0.25,1.11) circle (2.2pt);
\node[above, green!55!black] at (0.25,1.11) {$z_I$};
\node[font=\small\bfseries] at (0.9,-2.6) {(b)};
\end{scope}
\begin{scope}[shift={(11.0,0)}]
\coordinate (v) at (-0.5,0);
\coordinate (ystar) at (0.5,0);
\fill[green!12] (0.5,0) -- (1.2,1.3) -- (2.8,1.5) -- (3.3,0.2) -- (2.0,-0.6) -- cycle;
\node[green!40!black] at (2.2,0.3) {$\mathcal{I}$};
\draw[dashed, blue!65!black, thick] (0.5,-1.3) -- (0.5,1.7);
\draw[black!70, thick] (v) -- (-0.2,0.5);
\draw[densely dotted, thick, black!55] (0.5,0) -- (-0.5,1) -- (-1.5,0) -- (-0.5,-1) -- cycle;
\draw[->, thick, blue!65!black] (0.5,-1.15) -- ++(-0.7,0) node[below, midway] {$\pi$};
\draw[dashed, green!55!black, thick] (0.5,0) -- (-0.2,0.5) -- (1.2,1.3);
\foreach \p in {(1.2,1.3),(2.8,1.5),(3.3,0.2),(2.0,-0.6)} \filldraw[green!55!black] \p circle (2.2pt);
\filldraw[red] (v) circle (2.2pt);
\node[below] at (v) {$z$};
\filldraw[green!55!black] (ystar) circle (2.2pt);
\node[right, xshift=2pt, yshift=4pt] at (ystar) {$y^*$};
\filldraw[orange!90!black] (-0.2,0.5) circle (2.2pt);
\node[above, green!55!black] at (-0.2,0.5) {$z_I$};
\node[font=\small\bfseries] at (0.9,-2.6) {(c)};
\end{scope}
\end{tikzpicture}
}
\caption{
Geometric illustration of the supporting-hyperplane check (Theorem~\ref{thm:light-check}). The point $y^* \in \mathcal{I}$ is the optimal $\ell_1$-projection of point $z$ onto $\mathcal{I} = \operatorname{conv}(\mathcal{V}_I)$, with dual variable $\pi$ defining a supporting hyperplane to $\mathcal{I}$ at $y^*$. The dotted rhombus is the $\ell_1$ ball of radius $f(z)$ around $z$.
\textbf{(a)} The newly added point $z_I$ satisfies $\pi^\top z_I \le \pi^\top y^*$, certifying that $y^*$ remains the optimal projection over $\mathcal{I}' = \operatorname{conv}(\mathcal{V}_I \cup \{z_I\})$. 
\textbf{(b)} The condition fails ($\pi^\top z_I > \pi^\top y^*$), yet $y^*$ happens to remain optimal over $\mathcal{I}'$. 
\textbf{(c)} The condition fails ($\pi^\top z_I > \pi^\top y^*$), and $y^*$ is no longer optimal over $\mathcal{I}'$.
}

\label{fig:light-check-illustration}
\end{figure}
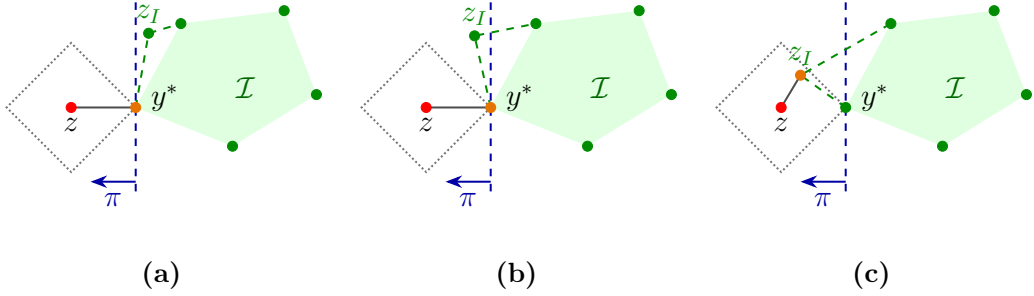

We refer to \eqref{eq:light-check} as the supporting-hyperplane check, since it tests whether $z_I$ lies on the same side of the hyperplane supporting $\mathcal{I}$ at $y^*$ as $\mathcal{I}$ itself (Figure~\ref{fig:light-check-illustration}).
It requires only one inner product per vertex, i.e., $O(n_z)$ operations, instead of solving Problem~\eqref{eq:f-def} again.
By Theorem~\ref{thm:light-check}, the cached $y^*$ and $\pi$ remain valid as long as the check passes, so the same check can be applied to the inner points added in later iterations.
Since the check is a sufficient condition, it never accepts an outdated distance, and whenever it fails, $f(v)$ is recomputed and $y^*$ and $\pi$ are updated.
For norms that are differentiable away from the origin (e.g., $\ell_2$), the check is also necessary (\ref{app:light-check-tight}), whereas for the $\ell_1$ and $\ell_\infty$ norms it is only sufficient, so a failed check may trigger a re-solve that turns out to be unnecessary, as in Figure~\ref{fig:light-check-illustration}(b).

\subsection{Algorithm}\label{sec:algorithm}

G-ORACLE is obtained by replacing Step~2 of ORACLE with the vertex enumeration, per-vertex convex projections, and supporting-hyperplane check described above, while Steps~1, 3, and~4 remain unchanged.
The resulting method is summarized in Algorithm~\ref{alg:geometric-trial-point}.
For each vertex $v \in \mathcal{V}_O$, the algorithm caches $f(v)$ together with the optimal solution $y^*(v)$ and dual vector $\pi(v)$ of Problem~\eqref{eq:f-def}.
At the start of each iteration, $\mathcal{V}_O$ is updated with the halfspace $\mathcal{H}$ added in Step~4 of the previous iteration.
Problem~\eqref{eq:f-def} is solved for every newly created vertex, whereas each persisting vertex is first tested with the supporting-hyperplane check against the inner point $z_I$ added in the previous iteration, and Problem~\eqref{eq:f-def} is solved only if the check fails.

\begin{algorithm}
\caption{G-ORACLE}\label{alg:geometric-trial-point}
\KwData{Initial $\mathcal{I} = \operatorname{conv}(\mathcal{V}_I)$, $\mathcal{O}$ (Step 1); tolerance $\delta > 0$}
\KwResult{$\mathcal{I}, \mathcal{O}$ converged to within $\delta$ of $\mathcal{Z}$}
$\mathcal{H} \gets \mathbb{R}^{n_z}$;\quad $\mathcal{V}_O \gets \varnothing$\;
\Repeat{$d \le \delta$}{
    \tcc{Step 2: geometric trial-point search}
    $\mathcal{V}_O' \gets \operatorname{vert}(\mathcal{O} \cap \mathcal{H})$ \tcp*{incrementally (Section~\ref{sec:vertex-enum})}
    \ForEach{$v \in \mathcal{V}_O' \setminus \mathcal{V}_O$}{
        solve Problem~\eqref{eq:f-def} for $v$ and cache $f(v)$, $y^*(v)$, $\pi(v)$\;
    }
    \ForEach{$v \in \mathcal{V}_O \cap \mathcal{V}_O'$}{
        \If{\eqref{eq:light-check} fails for $v$}{
            solve Problem~\eqref{eq:f-def} for $v$ and cache $f(v)$, $y^*(v)$, $\pi(v)$\;
        }
    }
    $\mathcal{O} \gets \mathcal{O} \cap \mathcal{H}$;\quad $\mathcal{V}_O \gets \mathcal{V}_O'$\;
    $z_O \gets \argmax_{v \in \mathcal{V}_O} f(v)$;\quad $d \gets f(z_O)$\;
    \tcc{end Step 2}
    \If{$d > \delta$}{
        solve Problem~\eqref{eq:nearest-feasible} and obtain $z_I$, $\sigma$ \tcp*{Step 3}
        $\mathcal{V}_I \gets \mathcal{V}_I \cup \{z_I\}$;\quad $\mathcal{H} \gets \mathbb{R}^{n_z}$ \tcp*{Step 4}
        \If{$z_O \ne z_I$}{
            $\mathcal{H} \gets \mathcal{H}_\mathrm{P}$\;
        }
    }
}
\end{algorithm}

\section{Numerical Experiments}\label{sec:numerical}

In this section we compare G-ORACLE with ORACLE on a European power system model.
Since the two methods differ only in how the trial point is found (Step~2), we compare the run time of this step on identical inner and outer approximations and check numerically that both return the same distance as predicted by Theorem~\ref{thm:vertex}.
Section~\ref{sec:setup} describes the model and the experimental setup, and Section~\ref{sec:results} presents the results.

\subsection{Experimental Setup}\label{sec:setup}

All experiments use an electricity-only configuration of PyPSA-Eur \citep{hoersch2018pypsaeur}, an open model of the European power system built on PyPSA \citep{brown2018pypsa}, with one node per country (34 nodes) and a 3-hourly resolution over one year (2920 time steps).
The resulting ESOM is an LP with about 2.72~million variables, 5.64~million constraints, and 11.2~million nonzeros.
We consider eight PoIs ($n_z = 8$), each defined as the total extendable capacity of one technology summed over all nodes, namely the power capacity of generators (\texttt{p\_nom} in PyPSA) or the energy capacity of stores (\texttt{e\_nom}).
Six PoIs correspond to generation technologies: onshore wind (onwind), near-shore offshore wind on fixed-bottom foundations (offwind-ac), deep-water offshore wind on floating foundations (offwind-float), fixed-tilt solar PV (solar), solar PV with horizontal single-axis tracking (solar-hsat), and combined-cycle gas turbines (CCGT).
The remaining two are storage energy capacities: battery (battery) and hydrogen (H2).

The cost slack is set to $\epsilon = 5\%$.
In Step~1, Problem~\eqref{eq:mga-subproblem} is solved for $w = \pm e_i$, $i = 1, \dots, n_z$, which yields the achievable range of each PoI (Table~\ref{tab:dimension-ranges}) and hence the initial box $\mathcal{O}$.
Together with the cost-optimal point, these $2n_z = 16$ solutions form the $2n_z + 1 = 17$ initial inner points.

\begin{table}[htbp]
\centering
\caption{Achievable range of each PoI, obtained in Step~1 and used to initialize $\mathcal{O}$ and $\mathcal{I}$.}
\label{tab:dimension-ranges}
\begin{tabular}{lrrl}
\toprule
PoI & Min & Max & Unit \\
\midrule
onwind & 235.8 & 534.6 & GW \\
offwind-ac & 48.7 & 159.3 & GW \\
offwind-float & 0.0 & 78.4 & GW \\
solar & 183.6 & 1{,}290.0 & GW \\
solar-hsat & 0.0 & 964.1 & GW \\
CCGT & 219.9 & 365.7 & GW \\
battery & 0.0 & 2{,}336.1 & GWh \\
H2 & 0.0 & 163.8 & TWh \\
\bottomrule
\end{tabular}
\end{table}

Since the PoIs have different units and scales (e.g., GW vs.\ TWh), distances in the original units would be dominated by the PoIs with the largest ranges, biasing the trial-point search toward them.
We therefore map each PoI affinely onto $[-1,1]$ using its achievable range and run both algorithms in these normalized coordinates, transforming the cuts obtained from the ESOM accordingly (\ref{app:normalization}).
Distances are measured with the $\ell_\infty$ norm, which, like the $\ell_1$ norm, admits ORACLE's MILP reformulation.
The convergence tolerance is set to $\delta_\text{rel} = 4\%$ of the achievable range of each PoI, which corresponds to $\delta = 2\,\delta_\text{rel} = 0.08$, since every normalized PoI spans an interval of width $2$.
Analogously, we define the relative distance $d_\text{rel} = d/2$, which expresses the directed Hausdorff distance $d$ from $\mathcal{O}$ to $\mathcal{I}$ in the normalized space in percent of the achievable range of each PoI, so that the algorithms stop once $d_\text{rel} \le \delta_\text{rel}$.

We use the implementation of the double description method in the C library cddlib \citep{cddlib}, accessed via its Python bindings pycddlib \citep{pycddlib}.
Although the double description method is incremental, cddlib does not support incremental updates, so we patched the library and its Python bindings to allow incremental vertex enumeration; the patched versions are available in our open-source implementation \citep{mga_compass}.
The library supports both double-precision floating-point and exact rational arithmetic, which we denote as cdd(float) and cdd(rational), respectively.
Rational arithmetic is exact and thus avoids the numerical errors of floating-point arithmetic, such as rounding errors that can cause nearly coincident vertices to be merged or missed, but it is considerably slower.

Since ORACLE and G-ORACLE differ only in Step~2, we compare this step on identical instances.
We run G-ORACLE with cdd(float) as the vertex enumeration method in Step~2 (Algorithm~\ref{alg:geometric-trial-point}).
The vertex update runs in a single process, while the per-vertex instances of Problem~\eqref{eq:f-def} are distributed across 8 parallel processes with one thread each.
At every iteration, we additionally solve Step~2 of ORACLE, i.e., Problem~\eqref{eq:trial-point-problem} via its MILP reformulation, on the same $\mathcal{I}$ and $\mathcal{O}$; its trial point is not used to proceed to Steps~3 and~4.
Its big-M constants are recomputed at every iteration from the KKT conditions and the axis-aligned bounding boxes of $\mathcal{O}$ and $\mathcal{I}$ in the normalized space (\ref{app:big-m}), which keeps them tight and strengthens the LP relaxation of the MILP.
Likewise, we run G-ORACLE's Step~2 with cdd(rational) on the same $\mathcal{I}$ and $\mathcal{O}$, with the same parallelization, as a reference to check whether cdd(float) yields the correct relative distance.
We denote these three variants of Step~2 as O-MILP for ORACLE and as G-float and G-rational for G-ORACLE with cdd(float) and cdd(rational), respectively.
For each variant, we measure its wall-clock time and stop evaluating it once $d_\text{rel} \le \delta_\text{rel}$ or it exceeds 20 minutes in a single iteration.

In Step~4, we add both the projection cut~\eqref{eq:projection-cut} and the cost cut~\eqref{eq:cost-cut}, the latter requiring one solve of Problem~\eqref{eq:value-fn} per iteration.
Hence, up to two halfspaces are added to $\mathcal{O}$ per iteration.

All algorithms are implemented in Python 3.14, and all optimization problems are solved with Gurobi 13 on a KVM virtual machine with 54 vCPUs (AMD EPYC, 2.0~GHz) and 125~GB of RAM.
The ESOMs are solved with the barrier method without crossover using 32 threads; all other settings are left at Gurobi's defaults.
The code and data used to generate the results are available in our open-source implementation \citep{mga_compass}.

\subsection{Results}\label{sec:results}

\begin{figure}
\centering
\includegraphics[width=\textwidth]{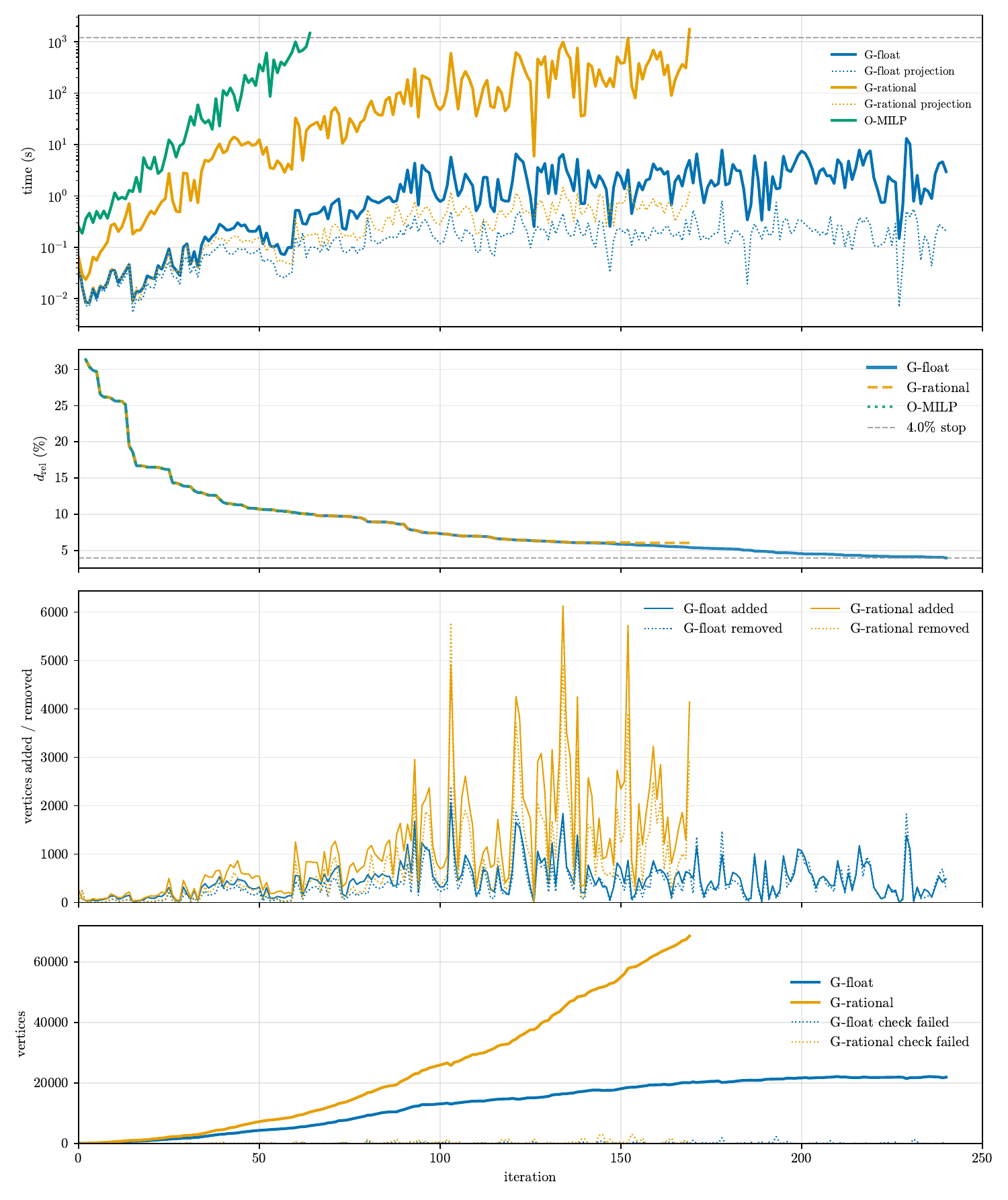}
\caption{
    \textbf{First}: Run time of Step~2 for O-MILP (ORACLE, MILP reformulation of the bilevel problem) and for G-float and G-rational (G-ORACLE, Algorithm~\ref{alg:geometric-trial-point}, with cdd(float) and cdd(rational), respectively). For G-float and G-rational, the dotted lines show the time spent on the projections, i.e., solving Problem~\eqref{eq:f-def}. Each variant is stopped once $d_\text{rel} \le \delta_\text{rel} = 4\%$ or it exceeds 20 minutes in a single iteration.
    \textbf{Second}: The relative distance $d_\text{rel}$.
    \textbf{Third}: Number of vertices added to and removed from $\mathcal{O}$ in each iteration.
    \textbf{Fourth}: Number of vertices of the outer approximation $\mathcal{O}$ after each iteration and number of vertices that fail the supporting-hyperplane check~\eqref{eq:light-check} in each iteration.
 }
\label{fig:stats-run-comparison}
\end{figure}

The first panel of Figure~\ref{fig:stats-run-comparison} shows the run time of O-MILP, G-float, and G-rational.
G-float reaches $d_\text{rel} \le \delta_\text{rel} = 4\%$ after 240 iterations, and its run time is at most 14~s per iteration.
In contrast, the run time of O-MILP increases rapidly with the number of iterations and exceeds 20 minutes at iteration 64 (1{,}469~s), when the relative distance (second panel) is still $10.02\%$.
Up to this iteration, the cumulative run time of G-float is only about 9~s.
At every iteration for which both were computed, O-MILP and G-float return the same distance $d$, which confirms Theorem~\ref{thm:vertex} numerically.

For comparison, each iteration of both ORACLE and G-ORACLE (Algorithms~\ref{alg:oracle} and~\ref{alg:geometric-trial-point}) requires two ESOM solves, namely Problem~\eqref{eq:nearest-feasible} in Step~3 and Problem~\eqref{eq:value-fn} for the cost cut in Step~4.
The former takes between 120 and 1{,}631~s, with a median of 241~s, and the latter between 114 and 277~s, with a median of 159~s.
Both are solved from scratch, since linopy \citep{Hofmann2023}, the optimization framework underlying PyPSA, does not support warm starts.
Hence, G-float accounts for only a small fraction of the run time per iteration, which is dominated by the ESOM solves.

G-rational exceeds 20 minutes at iteration 169 (1{,}746~s); up to this iteration, the cumulative run time of G-float is about 210~s.
The run time of both G-float and G-rational is correlated with the number of vertices added to and removed from $\mathcal{O}$ in each iteration (third panel).
Although G-float and G-rational should in theory produce the same vertex set, G-rational reports considerably more vertices (fourth panel), and the ratio of the two grows from 1.7 at iteration 50 to 3.4 at iteration 169 (68{,}606 vs.\ 20{,}121).
Even so, both remain far below the worst-case bound. With 354 halfspaces at iteration 169, i.e., the 16 of the initial box and two cuts per iteration, the upper bound theorem allows up to $V(354, 8) \approx 6.2 \times 10^8$ vertices by~\eqref{eq:upper-bound}, about 9{,}000 times as many as G-rational reports.

The mismatch between the vertex sets of G-float and G-rational affects the relative distance as well, although only in later iterations. Up to iteration 126, both variants return the same value within $6 \times 10^{-7}$ percentage points (pp), except at iteration 95 ($7.50\%$ vs.\ $7.61\%$).
Between iterations 127 and 136, they occasionally differ by up to 0.07~pp, and from iteration 137 on, they differ at every iteration.
In all 40 iterations in which they differ, G-rational returns the larger relative distance.
While the relative distance of G-rational stagnates at $6.11\%$ from iteration 136 to 152 and at $6.06\%$ from iteration 153 to 169, that of G-float keeps decreasing to $5.42\%$, so that the difference grows to 0.64~pp at iteration 169.
This suggests that cdd(float) merges or drops vertices of $\mathcal{O}$, which can include the vertex furthest from $\mathcal{I}$.
Since the trial points of G-float never target this vertex, the resulting cuts do not cut it off, and the relative distance of G-rational stagnates.
In this case, G-float underestimates $d$.

The fourth panel also shows the number of vertices that fail the supporting-hyperplane check~\eqref{eq:light-check} in each iteration.
Only for these vertices and for newly added vertices must Problem~\eqref{eq:f-def} be solved again; for all other vertices that survive the cut, the distance computed in a previous iteration remains valid (Theorem~\ref{thm:light-check}).
Summed over the 240 iterations of G-float, 47{,}026 of the 3{,}163{,}114 checks performed on surviving vertices fail, i.e., 1.49\%.
Hence, the supporting-hyperplane check avoids recomputing the distance in the remaining 98.51\% of the cases, so the time spent on the projections remains small compared with the vertex update (first panel).
Moreover, for 46{,}965 (99.87\%) of the failed checks, the recomputed distance is strictly smaller than the previous one.
Thus, although the check is only sufficient for the $\ell_\infty$ norm, only 61 (0.13\%) of the re-solves it triggers are unnecessary.

\section{Conclusion}\label{sec:conclusion}

We presented G-ORACLE, a geometric variant of the ORACLE method for exploring the near-optimal set of ESOMs.
It rests on the observation that the trial point sought by ORACLE's bilevel problem is always a vertex of the outer approximation, which allows the bilevel problem to be replaced by incremental vertex enumeration and independent per-vertex convex projections.
A supporting-hyperplane check, based on the dual vector of the projection problem, further avoids unnecessary recomputation of these projections as the inner approximation grows.
Since only the trial-point search changes, G-ORACLE retains ORACLE's convergence guarantee.

In our numerical experiments on a PyPSA-Eur model with eight PoIs, the trial-point search of ORACLE exceeds 20 minutes at iteration 64, when the relative distance is still 10.02\%, and its run time increases steeply with the number of iterations.
In contrast, the trial-point search of G-ORACLE with floating-point vertex enumeration takes at most 14~s per iteration throughout the 240 iterations needed to reach a relative distance of 4\%.
At every iteration for which both were computed, ORACLE and G-ORACLE return the same distance, which confirms Theorem~\ref{thm:vertex} numerically.
Compared with the two ESOM solves per iteration, with median run times of 241~s and 159~s, the trial-point search of G-ORACLE is negligible, so that the run time per iteration is dominated by the ESOM solves.

The double description method used for incremental vertex enumeration is considerably faster in floating-point than in exact rational arithmetic and returns nearly the same distance in early iterations, but a consistently smaller one in later iterations.
This suggests that floating-point arithmetic may drop the vertex furthest from the inner approximation, in which case it underestimates the distance.
This is consistent with the known numerical sensitivity of the double description method in floating-point arithmetic \citep{cddlib}.
Adjusting the numerical tolerances of the double description method and the feasibility tolerance of the LP solver is worth investigating to reduce this mismatch.

The run time of G-ORACLE's trial-point search depends on the number of vertices of the outer approximation, which is $O(m^{\lfloor n_z/2 \rfloor})$ in the worst case, where $m$ is the number of cuts in the outer approximation and $n_z$ is the number of PoIs (Section~\ref{sec:vertex-enum}).
Although the number of vertices remained several orders of magnitude below this bound in our numerical example with $n_z = 8$, it is worth investigating how it behaves for other ESOMs and larger $n_z$.


G-ORACLE allows several trial points to be processed per iteration. Since the distances of all vertices of $\mathcal{O}$ are available, Problem~\eqref{eq:nearest-feasible} can be solved in parallel for the $k$ vertices furthest from the inner approximation, adding $k$ inner points and up to $k$ cuts per iteration.
A caveat is that these vertices may lie close to each other, so the resulting cuts may be largely redundant and convergence may not accelerate in proportion to the additional computation.
This can be mitigated by selecting, among the vertices whose distance exceeds $\delta$, only those that are sufficiently far apart from each other.

Overall, by removing the bilevel bottleneck, G-ORACLE makes ORACLE's convergence guarantee affordable for ESOMs, leaving the ESOM solves as the only significant computational cost.

\appendix

\section{Tightness of the Supporting-Hyperplane Check}\label{app:light-check-tight}

Let $z$, $y^*$, $\pi$, $z_I$, and $\mathcal{I}'$ be as in Theorem~\ref{thm:light-check}, with $z \ne y^*$. Recall that $\pi$ is a subgradient of $\|\cdot\|$ at $u \in \mathbb{R}^{n_z}$, written $\pi \in \partial\|\cdot\|(u)$, if $\|w\| \ge \|u\| + \pi^\top(w - u)$ for all $w \in \mathbb{R}^{n_z}$, or equivalently, if $\|\pi\|_* \le 1$ and $\pi^\top u = \|u\|$. By dual feasibility and strong duality (proof of Theorem~\ref{thm:light-check}), the optimal dual vector $\pi$ of \eqref{eq:f-def} therefore satisfies $\pi \in \partial\|\cdot\|(z - y^*)$.

\begin{proposition}\label{prop:light-check-tight}
If $\|\cdot\|$ is differentiable at $z - y^*$ (e.g., $\|\cdot\|_p$ for $1 < p < \infty$), then \eqref{eq:light-check} is also necessary: $y^*$ minimizes $\|z - \cdot\|$ over $\mathcal{I}'$ if and only if $\pi^\top z_I \le \pi^\top y^*$.
\end{proposition}
\begin{proof}
Sufficiency is Theorem~\ref{thm:light-check}. For necessity, differentiability makes $\pi$ the unique subgradient of $\|\cdot\|$ at $z - y^*$, so $\varphi(y) = \|z - y\|$ is differentiable at $y^*$ with $\nabla\varphi(y^*) = -\pi$. The first-order optimality condition for minimizing the convex, differentiable $\varphi$ over the convex set $\mathcal{I}'$ states that $y^*$ is a minimizer if and only if $\nabla\varphi(y^*)^\top(y - y^*) \ge 0$ for all $y \in \mathcal{I}'$, i.e., $\pi^\top y \le \pi^\top y^*$ for all $y \in \mathcal{I}'$; taking $y = z_I$ gives \eqref{eq:light-check}.
\end{proof}

This tightness fails for $\ell_1$ and $\ell_\infty$, which are not differentiable at every point; for example, $\|\cdot\|_1$ is non-differentiable wherever a component of $z - y^*$ vanishes, and $\|\cdot\|_\infty$ wherever its maximum absolute component is attained more than once. At such points $\partial\|\cdot\|(z-y^*)$ is a set rather than a single vector, so $y^*$ can remain optimal over $\mathcal{I}'$ via some other subgradient in that set even when the particular $\pi$ obtained from solving \eqref{eq:f-def} fails \eqref{eq:light-check}.

\section{Normalization of the PoIs}\label{app:normalization}

Let $S = \operatorname{diag}(z_{ub} - z_{lb})$, where $z_{lb}$ and $z_{ub}$ are the bounds of the initial box $\mathcal{O}$ obtained in Step~1.
Each point $z \in \mathbb{R}^{n_z}$ is mapped to normalized coordinates by
\begin{equation*}
\tilde z = 2S^{-1}(z - z_{lb}) - \mathbf{1},
\end{equation*}
which maps the initial box onto $[-1,1]^{n_z}$. Its inverse is $z = \tfrac{1}{2}S\tilde z + \bar z$, where $\bar z = \tfrac{1}{2}(z_{lb} + z_{ub})$ is the center of the box.
Since the map is affine and invertible, it maps polytopes to polytopes and preserves convexity, so both algorithms apply unchanged in the normalized coordinates.

Problem~\eqref{eq:nearest-feasible} is solved in the original coordinates, with $z_O$ mapped back by the inverse map and the residual measured by the weighted norm $\|2S^{-1}r\|$.
Since $\tilde z_O - \tilde z = 2S^{-1}(z_O - z)$ for any $z$, this weighted norm equals the distance in normalized coordinates, so $z_I$ is the nearest point of $\mathcal{Z}$ to $z_O$ in the normalized space.
The projection cut $\mathcal{H}_\mathrm{P}$ and the cost cut $\mathcal{H}_\mathrm{C}$ are obtained from the ESOM in the original coordinates, as halfspaces $a^\top z \le a_0$.
Substituting $z = \tfrac{1}{2}S\tilde z + \bar z$ gives the equivalent halfspace
\begin{equation*}
\left(\tfrac{1}{2}Sa\right)^\top \tilde z \le a_0 - a^\top \bar z
\end{equation*}
in normalized coordinates, so each component of the normal vector is scaled by half the range of the corresponding PoI, and the right-hand side is shifted by the value of $a^\top z$ at the center of the box.

\section{Big-M Constants of the Bilevel Reformulation}\label{app:big-m}

For the $\ell_\infty$ norm, let $\mathcal{V}_I = \{u_1, \dots, u_k\}$ and $\mathcal{O} = \{z \in \mathbb{R}^{n_z} \mid A_O z \le b_O\}$. The lower-level problem of \eqref{eq:trial-point-problem} for a fixed $z \in \mathcal{O}$ is the LP
\begin{equation*}
\begin{aligned}[t]
f(z) = \min_{t \in \mathbb{R},\, \lambda \in \mathbb{R}^k} \quad & t \\
\text{s.t.} \quad & z - \textstyle\sum_{i=1}^k \lambda_i u_i \le t\mathbf{1}, \quad (\alpha^+) \\
& -z + \textstyle\sum_{i=1}^k \lambda_i u_i \le t\mathbf{1}, \quad (\alpha^-) \\
& \mathbf{1}^\top \lambda = 1, \quad (\eta) \\
& \lambda \ge 0, \quad (\beta)
\end{aligned}
\end{equation*}
in which the dual vector $\pi$ of Problem~\eqref{eq:f-def} corresponds to $\alpha^+ - \alpha^-$.
Its KKT conditions consist of primal feasibility, dual feasibility $\alpha^\pm, \beta \ge 0$, stationarity
\begin{equation*}
\textstyle\sum_{j=1}^{n_z} (\alpha^+_j + \alpha^-_j) = 1, \qquad \eta - \beta_i - u_i^\top(\alpha^+ - \alpha^-) = 0, \quad i = 1, \dots, k,
\end{equation*}
and the complementarity conditions $\alpha^\pm_j q^\pm_j = 0$ and $\beta_i \lambda_i = 0$, where $q^\pm_j = t \mp (z - \sum_i \lambda_i u_i)_j \ge 0$ are the primal slacks. Each complementarity condition is linearized with a binary variable, e.g., $\alpha^+_j \le M_D \theta^+_j$ and $q^+_j \le M_P(1 - \theta^+_j)$ with $\theta^+_j \in \{0,1\}$, where $M_P$ bounds the primal variables $q^\pm_j$ and $\lambda_i$, and $M_D$ bounds the dual variables $\alpha^\pm_j$ and $\beta_i$.

\textbf{Primal constant.} Since $\lambda$ is a vector of convex weights, $\lambda_i \le 1$. Let $[z^O_{lb}, z^O_{ub}]$ be the axis-aligned bounding box of $\mathcal{O}$, obtained from $2n_z$ LPs over $\mathcal{O}$, and $[z^I_{lb}, z^I_{ub}]$ that of $\mathcal{I}$, given by the componentwise minima and maxima of $u_1, \dots, u_k$. Since $z \in \mathcal{O}$ and every point of $\mathcal{I}$ lies in $[z^I_{lb}, z^I_{ub}]$,
\begin{equation*}
t^* \le T_{\max} := \max_{j} \max\left(z^O_{ub,j} - z^I_{lb,j},\; z^I_{ub,j} - z^O_{lb,j}\right),
\end{equation*}
and each slack satisfies $q^\pm_j \le t^* + T_{\max} \le 2T_{\max}$. These bounds need only hold at points satisfying the KKT conditions. Since the lower-level problem is an LP, its KKT conditions are sufficient for optimality, so every such point has $t = t^* = f(z)$, and no explicit bound on $t$ is required in the MILP. Hence $M_P = \max(1,\, 2T_{\max})$, where 1 bounds $\lambda_i$ and $2T_{\max}$ bounds $q^\pm_j$.

\textbf{Dual constant.} The first stationarity condition and $\alpha^\pm \ge 0$ give $\alpha^\pm_j \le 1$ and $\|\alpha^+ - \alpha^-\|_1 \le 1$. Since $\mathbf{1}^\top\lambda = 1$, some $\lambda_l > 0$, and thus $\beta_l = 0$, so the second stationarity condition and H\"older's inequality give $|\eta| = |u_l^\top(\alpha^+ - \alpha^-)| \le \|u_l\|_\infty \|\alpha^+ - \alpha^-\|_1 \le R_{\mathrm{in}} := \max_{i,j} |u_{ij}|$. Substituting into the same condition for every $i$ yields $\beta_i \le 2R_{\mathrm{in}}$. Hence $M_D = \max(1,\, 2R_{\mathrm{in}})$, where 1 bounds $\alpha^\pm_j$ and $2R_{\mathrm{in}}$ bounds $\beta_i$.

Both constants are recomputed at every iteration from the current $\mathcal{O}$ and $\mathcal{V}_I$ in the normalized space. Any larger values are also valid but weaken the LP relaxation of the MILP.

\bibliographystyle{elsarticle-harv}
\bibliography{refs}

\end{document}